\documentclass[11pt]{amsart} 

\usepackage{pdfsync}
\usepackage{lmodern}
\usepackage{amsthm,amsmath,amssymb,amscd,graphicx,enumerate, stmaryrd,xspace,verbatim, epic, eepic,color,url}
\usepackage{tikz}
\usepackage{pgfplots}
\usepackage{tikz-3dplot}

\newtheorem{thm}{Theorem}[section]
\newtheorem*{thm*}{Theorem}

\newtheorem{prop}[thm] {Proposition}     
\newtheorem{lemma} [thm]{Lemma}

\usepackage{hyperref}
\usepackage{pgf}

\newtheorem*{conjecture*}{Conjecture}
\newtheorem*{remark*}{Remark}
\newtheorem{question}{Question}

\theoremstyle{definition}

\newtheorem*{exercise*}{Exercise}

\newtheorem{defi}[thm] {Definition}

\newtheorem{remark} [thm]{Remark}

\usepackage{amsmath}
\usepackage{amssymb,amsthm}
\usepackage{amscd}
\usepackage{enumerate}
\input{xy}
\xyoption{all}

\newcounter{prg}[section]

\newcommand{\ov}{\overline}

\newcommand{\la}{\longrightarrow}

\newcommand{\A}{\mathbb{A}}
\newcommand{\C}{\mathbb{C}}
\newcommand{\F}{{\mathbb{F}}}
\newcommand{\G}{{\mathbb{G}}}
\newcommand{\N}{\mathbb{N}}
\newcommand{\PP}{\mathbb{P}}

\newcommand{\Z}{\mathbb{Z}}

\newcommand{\cE}{\mathcal {E}}

\newcommand{\Hyp}{\operatorname{Hyp}}

\newcommand{\mult}{{\operatorname{mult}}}

\def\md{\underline{d}}

\begin{document}

\title{Exceptional loci in algebraic surfaces}
\author{Lucia Caporaso and Amos Turchet}
\address[Caporaso]{Dipartimento di Matematica e Fisica\\ Universit\`{a} Roma Tre \\ Largo San Leonardo Murialdo \\I-00146 Roma\\  Italy }\email{lucia.caporaso@uniroma3.it}
\address[Turchet]{Dipartimento di Matematica e Fisica\\ Universit\`{a} Roma Tre \\ Largo San Leonardo Murialdo \\I-00146 Roma\\  Italy }\email{amos.turchet@uniroma3.it}
\begin{abstract}
	We study the algebraic exceptional set for surfaces $(S,B)$ of log general type, when $B$ has at least three irreducible components; we prove that in most cases it is finite or empty.
\end{abstract}
\maketitle

\tableofcontents

\section{Introduction}

We work over $\C$.
Let $S$ be a smooth connected projective surface, and let $B \subset S$ be a reduced projective curve of the form $B=  B_1\cup \ldots \cup B_n$, with $B_i$ integral curves such that for all $i\neq j$ every point in $B_i\cap B_j$ is a node of $B$. The main focus of this article is Lang's (algebraic) exceptional set, i.e.

\begin{equation} \label{eq:exc_set}
	\cE(B) := \{ C \subset S : C \text{ integral curve},\  2g(C) - 2 + \lvert \nu_C^{-1}(C \cap B) \rvert \leq 0 \},
\end{equation}
where   $\nu:C^{\nu}\to C$ is  the normalization and $C$ is projective.
\begin{remark}
	When $B$ is ample, as it will be the case for most of our applications, then $C \cap B$ is not empty. Hence  if $C \in \cE(B)$  then  $C$ is necessarily rational.
\end{remark}

According to Lang and Vojta conjectures (see for example \cite{Lan86, Vojta87}), the union of the curves in $\cE(B)$ should coincide with the arithmetic exceptional set (the largest set containing all but finitely many integral points for every finite extension of the base field, when $S$ is defined over a number field), and with the analytic exceptional set (the Zariski closure of the union of all the images of non constant holomorphic maps $\C \to S \smallsetminus B$).  Moreover, if the pair $(S,B)$ is of log general type, the set $\cE(B)$ is conjectured to be finite.

The main result of this article,   confirming the above conjecture, is the description of the set $\cE(B)$ when $B$ has  at least  three   very ample  irreducible  components, generalizing some results obtained in \cite{CT} for $S=\PP^2$. The case of Hirzebruch surfaces $S=\F_n$ is treated in \cite{Wei}.

The following is a summary of our results, with the case $n=3$ being the most interesting one.

\begin{thm}\label{th:main}
	Let $S$ and $B = \cup_{i=1}^n B_i$ be as above.  The following hold:
	\begin{enumerate}
		\item if $n \geq 5$ then  $\cE(B)$ is empty;
		\item   if $n = 4$ and $B_i$ is very ample for every $i$, then $\cE(B)$ is finite, and   is empty if at least three   components of $B$ are not primitive;
		\item
		      if $n=3$,\  $B_i$ is very ample  for every $i$, and at least  two   components of $B$ are not primitive,
		      then $\cE(B)$ is finite, and   is empty if all components of $B$ are not primitive.

	\end{enumerate}
\end{thm}

In the statement by ``primitive'' we mean not multiple of very ample.
The hypotheses  of the  theorem  ensure  that the pair $(S,B)$ is indeed of log general; cf. Remark~\ref{rk-gt}.

For the cases $n=4,5$ see Proposition~\ref{nC}.  For $n=3$ see Theorems~\ref{3Cnonprim} and \ref{thm-nonprim}.
To prove our results we   drop the genus assumption and    consider the following larger set
\begin{equation}
	\label{def-hyp}
	\Hyp(B,2):=\{C\subset S:\   C \text{ integral curve},\  |\nu ^{-1}(C\cap B)|\leq 2\}\supset \cE(B).
\end{equation}
We  then prove a stronger statement,   namely we prove that Theorem~\ref{th:main} holds   with $\Hyp(B,2)$ instead of $\cE(B)$.
We thus prove the finiteness of $\Hyp(B,2)$; by contrast, if $B$ has only two irreducible components, we have examples where $\Hyp(B,2)$ is infinite and strictly larger then $\cE(B)$.

The subsets of $S$ made by the curves in $\Hyp(B,2)$ and $\cE(B)$ will be denoted as follows
$$
	[\Hyp(B,2)]:=\bigcup_{C\in \Hyp(B,2)}C\subset S ,\quad \quad \quad [\cE(B)]:=\bigcup_{C\in \cE(B)}C\subset S.
$$
We abuse terminology and refer to $[\cE(B)]$ as the (algebraic) exceptional set.

The exceptional set, particularly in connection with  Vojta conjecture \cite[Conj. 3.4.3]{Vojta87}, has been considered in several articles; see  for instance  \cite{CZGm},\cite{Wang_etal},\cite{CRY}, and references therein. These papers establish the geometric form of Vojta conjecture   and, as a consequence, deduce various properties of the exceptional set. Our Theorem \ref{th:main} applies in some situations which are not covered by such  earlier results; it applies to surfaces that are not ramified covers of $\G_m^2$, and we impose no generality assumptions on   the surface $S$ or on the divisor $B$.
On the other hand, our normal crossings and primitivity assumptions  prevent Theorem \ref{th:main} from applying to  various cases (as in the above   papers)   where the geometric Vojta  conjecture is  known to hold.

\subsection{Brody Hyperbolicity}
Theorem \ref{th:main} also provides many examples of Brody hyperbolic surfaces where the boundary divisor has (at least) three irreducible components.

Recall that a complex analytic space $Y$ is called \emph{Brody hyperbolic} if every holomorphic map $\C \to Y$ is constant. The Green-Griffiths-Lang Conjecture predicts that, if $X$ is a complex algebraic variety of log general type, there exists a proper closed subset $Z$, such that if $f: \C \to X$ is a non-constant holomorphic map, then $f(\C) \subset Z$. In this case we say that $X$ is \emph{Brody hyperbolic modulo} $Z$. In particular, $X$ is Brody hyperbolic when $X$ is hyperbolic modulo $Z = \emptyset$.

In our setting the space under consideration is the complement $S \smallsetminus B$, for $S$ a smooth complex projective surface and $B = B_1 \cup \dots \cup B_n$ a reduced projective curve such that every point of $B_i \cap B_j$ is a node of $B$, for every $i \neq j$.
The irreducible components $B_i$ are called \emph{numerically parallel} if there exist $a_i \in \Z$, $i=1,\dots,n$, such that $a_i B_i$ is numerically equivalent to $a_j B_j$, for every $i \neq j$.

If $n\geq 3$, the pair $(S,B)$ is of log general type, and the components $B_i$ are numerically parallel, then \cite[Theorem 1]{RW} of Ru and Wang (extending work of Noguchi Winkelman and Yamanoi in \cite{NWY}) implies that the image of every non constant holomorphic map $\C \to S \smallsetminus B$ is an algebraic curve. This implies that the analytic exceptional set coincides with the algebraic exceptional set $[\cE(B)]$.
Therefore we obtain the following consequence of our Theorem \ref{th:main}.

\begin{thm}\label{th:hyperbolic}
	Let $S$ and $B = \cup_{i=1}^n B_i$ be as above. Assume that the irreducible components $B_i$ of $B$ are numerically parallel and very ample. 
	The following hold:
	\begin{enumerate}
		\item if $n \geq 5$ then $S \smallsetminus B$ is Brody hyperbolic;
		\item if $n=4$ then $S \smallsetminus B$ is Brody hyperbolic modulo $[\cE(B)]$ and  if at least $3$ of the components of $B$ are not primitive then $S \smallsetminus B$ is Brody hyperbolic;
		\item if $n=3$ and at least two among the components of $B$ are not primitive then $S \smallsetminus B$ is Brody hyperbolic modulo $[\cE(B)]$, and if all components of $B$ are not primitive then $S \smallsetminus B$ is Brody hyperbolic.
	\end{enumerate}
\end{thm}

\subsection*{Acknowledgements.}
We are grateful to Wei Chen for useful comments.
LC is partially supported by    PRIN 2022L34E7W,  Moduli spaces and birational geometry.
AT is partially supported by   PRIN  2022HPSNCR: Semiabelian varieties, Galois representations and related Diophantine problems, 
and is a member of the INdAM group GNSAGA.

\section{Preliminaries}

In our paper the word ``curve"  stands for ``projective curve".

As before $S$ denotes a smooth connected projective surface, $C$ denotes an integral curve contained in $S$, and $\nu:C^{\nu}\to C$ the normalization.

For a point  $q\in C\subset S$, we     denote by  $\sigma:S'\to S$ be the blow-up  at $q$, by $E$ the exceptional divisor and by $C'$ the strict transform of $C$. If $q$ is a unibranch point  of $C$ (i.e. $|\nu^{-1}(q)|=1$), we denote by $q'\in C'$ the point lying over $q$.

Let $B,C\subset S$ be two  integral  curves and $q\in B\cap C$. We denote by $(C\cdot B)_q$ the intersection multiplicity  at $q$.
We say that $B$ and $C$ are {\it transverse at} $q$
if $B'\cap C'\cap E=\emptyset$.  If $q$ is   unibranch   for both $B$ and $C$, we say that $B$ and $C$ are  {\it tangent at} $q$  if they  are not transverse at $q$.

We will frequently use the following elementary lemmas (whose proofs are included for completeness).
\begin{lemma}
	\label{lm-int}
	Let $B,C\subset S$ be two  integral   curves and $q\in B\cap C$.
	Then
	$$(C\cdot B)_q\geq \mult_q(B) \mult_q(C)$$ with equality if and only if $B$ and $C$ are transverse   at $q$.

	Let  $A,B, C\subset S$ be integral curves with a unibranch  point at $q$. If $A$ and $B$ are tangent  at $q$, and $B$ and $C$ are tangent  at $q$, then $A$ and $C$ are  tangent at $q$.
\end{lemma}

\begin{proof}
	Let $\sigma:S'\to S$ be the blow-up  at $q$ and  $A', B',C'\subset S'$ the strict transforms of $A$, $B$,  $C$.
	Then  $C'=\sigma^*C- \mult_q(C)E$ and $B'=\sigma^* B- \mult_q(B)E$, hence    $$(C'\cdot B')=(C\cdot B)- \mult_q(B)  \mult_q(C) = (C\cdot B)_q+r- \mult_q(B) \mult_q(C) $$  where $r=(C\cdot B)-(C\cdot B)_q\geq 0$.
	On the other hand  $$(C'\cdot B')=(C'\cdot B')_{E}+ r$$ where $(C'\cdot B')_E$ denotes the intersection multiplicity along points on the exceptional divisor, $E$. Combining the two identities we obtain  $$(C\cdot B)_q - \mult_q(B)  \mult_q(C)=(C'\cdot B')_E. $$ As  $(C'\cdot B')_E\geq 0$ we get  $(C\cdot B)_q\geq  \mult_q(B) \mult_q(C)$, with equality  if and only if $(C'\cdot B')_E=0$, i.e. if and only if
	$C$ and $B$ are transverse at $q$.

	Assume $q$ is a unibranch point for $A$, $B$, and $C$. Let $q'\in B'$ be the unique point lying over $q$.
	If $A$ and $B$ are tangent at $q$, then $A'$ and $B'$ intersect along $E$, hence $q'\in A'$; similarly, if   $C$ and $B$ are tangent at $q$, then $q'\in C'$. Hence $q'\in C'\cap A'$, hence $C$ and $A$ are not transverse at $q$.
\end{proof}

\begin{lemma}
	\label{lm-tanE}
	Let $q$ be a unibranch  point of $C\subset S$. Let $C'\subset S'$ be the strict transform of $C$ under the blow-up at $q$ and $q'\in C'$   the point lying over $q$.

	Then $\mult_{q'}(C')\leq \mult_{q}(C)$ with   equality  if and only if $C'$ is transverse to $E$ (at $q'$).
\end{lemma}

\begin{proof}   We have
	$$
		\mult_{q}(C)=(C'\cdot E)_{q'}\geq  \mult_{q'}(C')\mult_{q'}(E)= \mult_{q'}(C'),
	$$
	by Lemma~\ref{lm-int}, equality holds if and only if $C'$ is transverse to $E$.  \end{proof}

\begin{remark}
	By   Lemma~\ref{lm-int},  we have
	$
		\mult_q(C)\leq (C\cdot B)
	$
	for any integral curve $B$ passing through $q$. If $\mult_q(C)\geq 2$ and $S=\PP^2$ then strict inequality always holds, but this does not generalize to other surfaces. For example, let $S$ be the blow-up of $\PP^2$ in one point, $p$.
	Let $C_0\subset \PP^2$ be a cubic containing $p$, having  a cusp in $q_0\neq p$, and let $B_0$ be the line through $q_0$ and $p$. Let $C$ and $B$ be the strict transforms of $B_0$ and $C_0$ in $S$, and let $q$ be the point lying over $q_0$.
	Since $(C_0\cdot B_0)_{q_0}=2$ and  $(C_0\cdot B_0)_{p}=1$ we have
	$
		2=\mult_q(C)=(C\cdot B)_q=(C\cdot B).
	$
\end{remark}

Let $H$ be a very ample line divisor on a   surface $S$; we say that $H$  is  {\it primitive} if it is not   a nontrivial multiple of a very ample  divisor.
Note that if $H$ satisfies $(C\cdot H)=1$ for some curve $C\subset S$, then $H$ is   primitive.

\begin{lemma}
	\label{lm0}
	Let $C,H\subset S$  be two curves with   $H$ very ample;   set $d=(C\cdot H)$. Then the following hold.
	\begin{enumerate}[(a)]
		\item
		      \label{lm01}
		      For any $q\in C$
		      either $d>\mult_q(C)$, or $d=\mult_q(C)=1$, hence  $H$ is primitive.

		\item
		      \label{lm02}
		      If $d\geq 2$ then  for any $q_1,q_2\in C$,
		      we have  $d\geq \mult_{q_1}(C)+ \mult_{q_2}(C)$, and if equality holds, then  $H$ is primitive.

	\end{enumerate} \end{lemma}

\begin{proof}
	We view  $S\subset  \PP^r$ embedded by the linear system $|H|$, so that $r=h^0(H)-1$ and  $\deg C=d$.
	The linear subspace spanned by $C$ is  $\langle C \rangle \cong \PP^l$ with $d\leq l\leq r$. If $l=1$ then $C$ is a line, hence   $d=1$ and
	$\mult_q(C)=1$. Assume $d\geq 2$.
	If $l=2$ then $\langle C \rangle \cong \PP^2$  where all  statements are well known. We prove \eqref{lm01}   by induction on $l$.
	Let $l\geq 3$.
	Then the general point   $p\in \langle C \rangle\setminus C$ does not lie on a    line passing through $q$ and through another point of $C$
	(indeed, the closure of the union of all such lines has dimension $2$). Fix such a point $p$ and let
	$$
		\pi:C\la \PP^{l-1}
	$$
	be the projection from $p$. Set $\ov{C}=\pi(C)$ and $\ov{q}=\pi (q)$. By construction,  $\deg \ov{C}=d$ and $\mult_{\ov{q}}(\ov{C})\geq \mult_q(C)$.
	By the induction hypothesis, either $\deg \ov{C}> \mult_{\ov{q}}(\ov{C})$, hence $d> \mult_q(C)$, or $\deg \ov{C}= \mult_{\ov{q}}(\ov{C})=1$, hence
	$d= \mult_{q}(C)=1$.     \eqref{lm01} is proved.

	For  \eqref{lm02}, set  $m_i=\mult_{q_i}(C)$ for $i=1,2$. As $H$ is very ample, up to replacing it with a linearly equivalent effective divisor   we can assume $q_1,q_2\in H$. Hence
	$$
		d=(H\cdot C)\geq (H\cdot C)_{q_1}+(H\cdot C)_{q_2}\geq m_1+m_2
	$$
	as claimed.
	If $m_1+m_2 =(H\cdot C)$   and $H=nA$ with $A$   very ample,
	then
	$$
		m_1+m_2  =(nA\cdot C)=n(A\cdot C)\geq n((A\cdot C)_{q_1}+(A\cdot C)_{q_2})\geq n(m_1+m_2)
	$$
	hence $n=1$. This shows that $H$ is primitive.
\end{proof}

\begin{lemma}
	\label{S2}
	Let $B,C\subset S$  be integral curves such that   $B$ is  very ample and  $(C\cdot B)\geq 2$.
	If
	$C\cap B=\{q\}$ and $B$ is smooth at $q$,  then $B$ is tangent to $C$ at $q$.
\end{lemma}

\begin{proof}
	Set $d=(B\cdot C)$. By contradiction, suppose that the proper transforms, $B'$ and $C'$ in the blow-up of $S$ at $q$ do not intersect over $q$. Hence they do not intersect at all, by hypothesis, i.e.
	$ (B'\cdot C') =0$. Now
	$$
		0= (B'\cdot C')=(B\cdot C)_q-\mult_q(B)\mult_q(C)=d-\mult_q(C).
	$$
	Hence $d=\mult_q(C)$, which, as $d\geq 2$, contradicts Lemma~\ref{lm0}\eqref{lm01}.
\end{proof}

\section{Hypertangency}

Let $C\subset S$ and $q\in C$.
The arithmetic genus, $p_a(C)$, of $C$ is computed  by  adjunction,
as follows:
$p_a(C)=(C\cdot (K_S+C))/2+1$.
The  geometric genus, $g(C)$,   satisfies the following inequality
$$
	g(C)\geq p_a(C)-\delta_C (q)
$$
with equality if and only if $q$ is the only singular point of $C$.
The following extends the second part of Theorem 2.2.1 in \cite{CT}.
\begin{prop}
	\label{prop-genus}
	Let $B, C\subset S$ be two   integral curves   such that $B$ is very ample and not primitive.
	Assume
	$B\cap C=\{q\}$, where    $q$ is smooth for $B$, and unibranch   of multiplicity $m$ for    $C$.
	Then
	$$
		\delta_C(q)\geq (m-1)((B\cdot C)-m)/2.
	$$
\end{prop}

\begin{proof}
	By hypothesis there exists a very ample $H$ such that
	$B\sim aH$ with $a\geq 2$.

	If $(C\cdot H)=1$ then $C$ is necessarily smooth (it is a line in the embedding determined by $H$). Hence we   assume $d:=(C\cdot H)\geq 2$.

	We denote by $S_1\to S$ the blow-up at $q$,   by $C^1,B^1$ and $E_1$ the proper transforms  of $C, B$  and  the exceptional divisor; we write $q^1\in C^1$ for the point  lying over $q$.
	Now we define inductively, for $i\geq 2$, the   blow-up of $S_i\to S_{i-1}$ at the unique point $q^{i-1}\in C^{i-1}$,  let
	$C^i,B^i, E_i\subset S_i$ be the proper transforms  of $C^{i-1}$ and $B^{i-1}$ and  the exceptional divisor, and   $q^i\in C^i$ the point  lying over $q^{i-1}$. We set
	$$
		h=\lceil ((B\cdot C)-m)/m\rceil.
	$$
	We claim that    for every $1 \leq i \leq h-1$ the curves  $C^i$ and $B^i$ are tangent  in $q^i$, and  $C^i$ has an $m$-fold point at $q^i$.

	As $H$ is very ample and $d\geq 2$, we have
	$$
		d>m.
	$$
	Let us treat the induction base.
	Since $B\cap C=\{q\}$ we have $$(B^1\cdot C^1)_{q^1}=(B^1\cdot C^1)=(B\cdot C)-m$$ hence
	$$
		(B^1\cdot C^1)_{q^1}= (aH\cdot C)- m=ad -m>am-m=(a-1)m\geq m$$
	(as $a\geq 2$). Since  $B^1$ is smooth at $q^1$, Lemma~\ref{lm-int} gives that $C^1$ and $B^1$ are tangent in $q^1$.
	On the other hand $B^1$ is not tangent at $E_1$, hence neither is $C^1$. By Lemma~\ref{lm-tanE} we get
	$\mult_{q^1}(C^1)=m$. The induction base is proved.

	Now suppose $C^{i-1}$ has an $m$-fold point at $q^{i-1}$,
	and $B^{i-1}$ and  $C^{i-1}$ are tangent in  $q^{i-1}$.  Therefore
	$$
		(B^i\cdot C^i)_{q^i}=(B^i\cdot C^i)=(B\cdot C)-im\geq (B\cdot C)-(h-1)m.$$
	Now,
	$$h=\lceil ((B\cdot C)-m)/m\rceil = \lceil (B\cdot C)/m-1\rceil<  (B\cdot C)/m
	$$
	hence
	$$
		(B^i\cdot C^i)_{q^i}>(B\cdot C)-((B\cdot C)/m-1)/m=m
	$$
	Hence $C^i$ and $B^i$   are tangent in $q^i$. Hence $C^i$ is   not tangent to the exceptional divisor $E_i$, because $B^i$ is not. Hence $\mult_{q^i}(C^i)=m$, by Lemma~\ref{lm-tanE}.
	The claim is proved.

	Since  $C^i$ has an $m$-fold point at $q^i$ for every $i=0,\ldots, h-1$, we have
	\begin{align*}
		\delta_C(q)\geq hm(m-1)/2=\lceil  ((B\cdot C)-m)/m \rceil m(m-1)/2 \\
		\geq \Bigl(((B\cdot C)-m)/m \Bigr)m(m-1)/2=((B\cdot C)-m) (m-1)/2.
	\end{align*}

\end{proof}

For a curve  $B\subset S$   the set of curves hyperbitangent  to $B$ was defined in the introduction:
$$
	\Hyp(B,2)=\{C\subset S:\   C \text{ integral},\  |\nu ^{-1}(B)|\leq 2\};
$$
for $d\in \N$  we set
$$
	\Hyp_d(B,2)=\{C\in  \Hyp(B,2):\  (C\cdot B)=d\}.
$$
Given $q,q'\in B$, we denote
$$
	\Hyp (B;q):=\{C \in \Hyp(B,2) : C \cap B = \{ q \} \}
$$
$$
	\Hyp (B;q,q'):=\{C \in \Hyp(B,2) : C \cap B = \{ q,q' \} \},
$$
and
$$
	\Hyp_d(B;q):=\{C\in  \Hyp(B;q):\  (C\cdot B)=d\}. $$

We write $\Hyp_{\geq d}$ when the degree condition ``$(C\cdot B)=d$'' is replaced by  ``$(C\cdot B)\geq d$".

If $B=\cup_{i=1}^nB_i$, for any $\md=(d_1,\ldots, d_n)\in \Z^n$ we set
$$
	\Hyp_{\md}(B,2):=\{C\in \Hyp (B,2):\  (C\cdot B_i)=d_i,\ \forall i=1,\ldots,n\}
$$
and
$
	\Hyp_{\md}(B,q):=\{C\in \Hyp (B;q):\  (C\cdot B_i)=d_i,\ \forall i=1,\ldots,n\}.
$


\begin{defi}
	\label{defi-nC}
	An {\it $n$C-curve} is a reduced   curve   $B=  B_1\cup \ldots \cup B_n$,
	with $B_i$  integral      such that  for all $i\neq j$ every point in $B_i\cap B_j$ is a node of $B$.
	We write  $N:=\cup_{i\neq j} B_i\cap B_j.$
\end{defi}
Recall that a  singular point $q$ on a curve $B$ is  called a node  if a local  analytic neighborhood of $q$ is isomorphic to a neighborhood of the origin of the  curve in $\A^2$ of  equation $xy=0$.
\begin{remark}
	By definition,   the components of an  $n$C-curve  are smooth at each of the  intersection points in $N$, where they meet only pairwise and  transversally.
\end{remark}

\begin{prop}
	\label{S1} Let
	$B=B_1\cup B_2$ be a 2C-curve with $B_1$ and $B_2$ very ample; let $q\in B_1\cap B_2$.

	\begin{enumerate}[(a)]
		\item
		      \label{S1a}
		      If $B_1$ and $B_2$ are not primitive, then $ \Hyp (B;q) =\emptyset $.
		\item
		      \label{S1b}
		      If $B_1$ is primitive and $B_2$ is not primitive, then    $|\Hyp  (B;q)|\leq 1.$
		      More precisely, there exists at most one $l_2\in \N$ such that $|\Hyp_{(1,l_2)}(B;q)|= 1$
		      and $\Hyp_{(d_1,d_2)}  (B;q)=\emptyset $ for every $(d_1,d_2)\neq (1,l_2)$.
		\item
		      \label{S1c}
		      If $B_1$ and $B_2$ are both primitive, then $|\Hyp(B,q)\setminus \Hyp_{(1,1)}(B,q)|\leq 2.$
	\end{enumerate}
\end{prop}
\begin{proof}
	Clearly,
	$$\Hyp_{(d_1,d_2)}(B;q)=\Hyp_{d_1}(B_1;q)\cap \Hyp_{d_2}(B_2;q).
	$$
	Let $C\in \Hyp_{(d_1,d_2)}(B;q)$.
	Let $S'\to S$ be the blow-up at $q$, $E$ the exceptional divisor, $C'$  and $B_i'$ the proper transforms of $C$ and $B_i$ with $i=1,2$.
	As $B_1$ and $B_2$ are transverse at $q$, we have $E\cap B_i'=\{q_i\}$ with $q_1\neq q_2$. Since $C'\cap E$ consists of only one point,
	up to exchanging $B_1$ and $B_2$  we have
	$q_2\in C'$ and  $q_1\not\in C'$. Therefore
	$ (C'\cdot B_1')=0$; now, as $B_1$ is smooth at $q$, we have
	$$
		0= (C'\cdot B_1')=(C'\cdot B_1')_{q_1}=(C\cdot B_1)_{q}-\mult_q(C)=(C\cdot B_1)-\mult_q(C).
	$$
	Hence $(C\cdot B_1)=\mult_q(C)$;  Lemma~\ref{lm0}\eqref{lm01}  implies  that  $B_1$ is primitive and $d_1=(B_1\cdot C)=1$.  This proves  \eqref{S1a}.

	Assume   that $B_1$ is primitive;   we just proved that  if  $\Hyp_{(d_1,d_2)}(B;q)$ is non empty  then $d_1=1$.
	In the projective embedding of $S$ determined by $B_1$, the curve $C$ is a line  passing through $q$ and
	such that $(C\cdot B_2)_q=(C\cdot B_2)=d_2$. If $d_2\geq 2$ then  $C$ is the line of $S$
	hypertangent to $B_2$ at $q$; it is clear that  such a line, if it   exists, is
	uniquely determined by $B_2$ and $q$ (and by the embedding of $S$ given by $B_1$), and so is  the intersection number $(C\cdot B_2)$. If $B_2$ is not primitive,   we necessarily have $d_2\geq 2$, hence  part  \eqref{S1b} is proved
	with  $l_2= (C\cdot B_2)$.

	If $B_2$ is primitive  we can    have $d_2=1$ (and $\Hyp_{(1,1)}(B,q)$ can be infinite, see the next remark). The set $\Hyp(B,q)\setminus \Hyp_{(1,1)}(B,q)$ may   contain two curves, the one in $\Hyp_{(1,l_2)}(B,q)$ described above, and the one (applying the previous argument  to $B_2$ in place of $B_1$) in $\Hyp_{(l_1,l_2)}(B,q)$  for a unique $l_1\in \N$. This concludes the proof of part  \eqref{S1c}.
\end{proof}

\begin{remark}
	In case $B_1$ and $B_2$ are both primitive, the set $\Hyp_{(1,1)}(B,q)$ can easily be infinite. For example,  if  $B_1$ and $B_2$ are  two lines in $\PP^2$ meeting at the point $q$,
	then $\Hyp_{(1,1)}(B,q)$ is the set of all lines through $q$.
\end{remark}

\section{Hyperbitangency}
\begin{prop}
	\label{nC}
	Let $B=\cup_{i=1}^nB_i\subset S$ be an  $n$C-curve.
	\begin{enumerate}
		\item
		      If $n\geq 5$ then $\Hyp(B,2)$ is empty.
		\item
		      If $n=4$ and $B_i$ is very ample for every $i=1,\ldots, 4$, then  $\Hyp(B,2)$ is finite, and it is empty if  least three among the components of $B$  are not primitive.
	\end{enumerate}
\end{prop}

\begin{proof}
	The first part is an obvious consequence of the fact that   every point of $B$ belongs to at most two irreducible components.

	Suppose $n=4$. If $C\in  \Hyp(B,2)$ then   we   have $C\in   \Hyp(B; p_{i,j},p_{h,k})$ with $p_{i,j}\in B_i\cap B_j$
	and $\{i,j,h,k\}=\{1,2,3,4\}$.
	Therefore
	$$
		C\in  \Hyp(B_i\cup B_j,p_{i,j})\cap \Hyp(B_h\cup B_k,p_{h,k}).
	$$
	By Proposition~\ref{S1}, the set $\Hyp(B_i\cup B_j,p_{i,j})$ is finite, and empty if $B_i$ and $B_j$ are not primitive. Hence $ \Hyp(B; p_{i,j},p_{h,k})$ is  finite, and if it is non-empty then  at least one between $B_i$ and $B_j$, and at least one between $B_h$ and $B_k$, must be primitive.
\end{proof}

\begin{thm}
	\label{3Cnonprim}
	Let $B=B_1\cup B_2\cup B_3\subset S$ be a  3C-curve such that  $B_i$ is very ample  and not primitive for $i=1,2,3$.
	Then $\Hyp(B,2)$ is empty.
\end{thm}

\begin{proof}
	By contradiction, let  $C\in  \Hyp(B,2)$. As
	$B_i$ is not primitive,  we have $d_i:=(C\cdot B_i)\geq 2$. Since the components of $B$ intersect only pairwise, we necessarily have $|C\cap B|=2$.
	Suppose that $C$ intersects one of the three curves, say $B_3$, away from $N$. Then $C$ must intersect $B_1$ and $B_2$ in a point $q\in B_1\cap B_2$, and
	$$
		C\in \Hyp_{(d_1,d_2)}(B_1\cup B_2;q).
	$$
	Hence, by Proposition~\ref{S1},   one between $B_1$ and $B_2$ is primitive, which is impossible.
	Therefore $C\cap B\subset N$. We can assume $C\cap B=\{p_{1,2},p_{1,3}\}$ with $p_{i,j}\in B_i\cap B_j$.
	Let us show that $C$ is tangent to $B_2$ at $p_{1,2}$  and tangent to $B_3$ at $p_{1,3}$.
	We have
	$$
		C\in  \Hyp(B_2;p_{1,2}).
	$$
	Since  $(C\cdot B_2)\geq 2$, Lemma~\ref{S2} implies that $B_2$ is tangent to $C$ at $p_{1,2}$. By a similar  reasoning,   $B_3$ is tangent to $C$ at $p_{1,3}$.

	As $C$ is tangent to $B_i$ at $p_{1,i}$ for $i=2,3$, by Lemma~\ref{lm-int} it is  transverse to $B_1$ at both points. Hence, as $B_1$ is smooth at both points, we have
	$$
		d_1=(C\cdot B_1)=(C\cdot B_1)_{p_{1,2}}+(C\cdot B_1)_{p_{1,3}}=\mult_{p_{1,2}}(C)+\mult_{p_{1,3}}(C).
	$$
	Lemma~\ref{lm0}\eqref{lm02} implies that $B_1$ is primitive, which is a contradicton.
\end{proof}

For an integer $m\geq 1$ we denote by $\Hyp_*^m(B;q)\subset \Hyp_*(B;q)$ the set of curves in $ \Hyp_*(B;q)$ having a point of multiplicity exactly $m$ at $q$.

\begin{prop}
	\label{3CS}
	Let $B=B_1\cup B_2\cup B_3\subset S$ be a  3C-curve with  $B_i$   very ample  for $i=1,2,3$.
	Assume $B_1$ primitive and $B_3$ not primitive. Then
	\begin{enumerate}[(a)]
		\item
		      \label{3CSa}
		      $\Hyp (B,2)=\cE(B)$.

		\item
		      \label{3CSb}
		      Assume  $\md=(d_1,d_2,d_3)\in \Z^3$ with $d_i\geq 2$   for all $i$.
		      \begin{enumerate}[(i)]
			      \item
			            If    $B_2$ is not primitive then $d_1=2$ and
			            $$\Hyp_{\md}(B,2)=
				            \bigcup_{\stackrel {p\in B_1\cap B_2}{q\in B_1\cap B_3}} \Hyp_{2}(B_1;p,q) \cap  \Hyp_{d_2}(B_2;p)\cap \Hyp_{d_3} (B_3;q).
			            $$
			      \item
			            If    $B_2$ is  primitive, then
			            $$\Hyp_{\md}(B,2)= \bigcup_{\{i,j\}=\{1,2\}}   \bigcup_{\stackrel {p\in B_i\cap B_2}{q\in B_i\cap B_3}} \Hyp_{d_i}(B_i;p,q) \cap  \Hyp_{d_j}^{d_i-1}(B_j;p)\cap \Hyp^{1}_{d_3} (B_3;q).
			            $$
		      \end{enumerate}
	\end{enumerate}
\end{prop}

\begin{proof}
	Let us show that   \eqref{3CSa} follows from \eqref{3CSb}. Let $C\in \Hyp_{\md}(B,2)$.
	If   $d_i=1$ for some $i$, then  $C$ is a line in the embedding in projective space determined by $B_i$, hence $C$ is rational. If $d_i\geq 2$ for  all $i$, we can apply \eqref{3CSb}. If  $B_2$ is not primitive,
	in the projective embedding  detemined by $B_1$, the curve $C$ is a conic, hence  is rational.
	If $B_2$ is   primitive, then for some $i\in \{1,2\}$  the embedding of $C$ given  by $B_i$    is a curve of degree $d_i$  with  a point of multiplicity $d_i-1$,  so it is  rational.

	To prove  \eqref{3CSb}, we first observe that for all $i\neq j$ and $p_{i,j}\in B_i\cap B_j$ we have
	$$\Hyp_{\geq 2}(B_i;p_{i,j})\cap  \Hyp_{\geq 2}(B_j; p_{i,j})=\emptyset.$$
	Indeed, by Lemma~\ref{S2}, any curve in the above intersection would have to be tangent to both $B_i$ and $B_j$ at $p_{i,j}$, which is not possible as $B_i$ and $B_j$ meet transversally.

	Let $C\in \Hyp_{\md}(B,2)$.  The preceeding observation implies   $C\cap B\subset N$. We claim that $|C\cap B_3|=1$. Indeed, if $|C\cap B_3|=2$ then $C\cap B=\{p_{1,3},p_{2,3}\}$ with  $p_{i,j}\in B_i\cap B_j$.
	Then $C\in \Hyp_{\geq 2}(B_1;p_{1,3})$, hence $C$ is not transverse to $B_1$ at $p_{1,3}$, hence it is transverse to $B_3$ at $p_{1,3}$;  the same  argument   switching $B_1$ with $B_2$ gives that $C$  is transverse to $B_3$ at $p_{2,3}$. Hence
	$$
		(C\cdot B_3)=\mult_{p_{1,3}}(C)+\mult_{p_{2,3}}(C),
	$$
	now Lemma~\ref{lm0}\eqref{lm02} implies that  $B_3$ is primitive, which is not possible.

	This argument shows that either  $C\cap B=\{p_{1,2},p_{1,3}\}$,   and
	\begin{equation}
		\label{eq-d1}
		d_1=\mult_{p_{1,2}}(C)+\mult_{p_{1,3}}(C)
	\end{equation}
	or   $C\cap B=\{p_{1,2},p_{2,3}\}$, with  $B_2$    primitive, and
	\begin{equation}
		\label{eq-d2}
		d_2=\mult_{p_{1,2}}(C)+\mult_{p_{2,3}}(C).
	\end{equation}
	Assume \eqref{eq-d1} holds.
	Set $m=\mult_{p_{1,3}}(C)$ so that $\mult_{p_{1,2}}(C)=d_1-m$.
	We claim that $m=1$.
	As $B_3$ is not primitive, we have $B_3=aH$ for some primitive, very ample $H$, and $a\geq 2$.
	If $(C\cdot H)=1$, then $C$ is a line with respect to $H$, hence $C$ is smooth and $m=1$.
	If  $(C\cdot H)\geq 2$
	we   apply Proposition~\ref{prop-genus} to $B_3$ and  $C$,  getting
	$$
		\delta_C(p_{1,3})\geq (m-1)((B_3\cdot C)-m)/2=(m-1)(a(C\cdot H)-m)/2.
	$$
	Now  set
	$$
		e =\min \{(C\cdot H), d_1\},
	$$
	so that
	\begin{equation}
		\label{eq-13}
		\delta_C(p_{1,3})\geq  (m-1)(ae -m)/2.
	\end{equation}
	Now, as $e$ is the degree of some embedding of $C$ in projective space, the arithmetic genus of $C$ satisfies
	\begin{equation}
		\label{eq-gene}
		p_a(C)\leq {{e-1}\choose{2}}.
	\end{equation}
	Moreover, as $p_{1,2}$ has multiplicity $d_1-m$,   we have
	\begin{equation}
		\label{eq-12}
		\delta_C(p_{1,2})\geq (d_1-m)(d_1-m-1)/2\geq (e-m)(e-m-1)/2.
	\end{equation}
	Assume, by contradiction, $m\geq 2$. From the above inequalities we obtain
	\begin{align*}
		g(C)\leq p_a(C)-\delta_C(p_{1,2})-\delta_C(p_{1,3}) & \leq \\
		{\binom{e-1}{2}}-(e-m)(e-m-1)/2-(m-1)(ae -m)/2      & =    \\
		{\binom{e-1}{2}}-(e^2-e-2me+ame-ae+2m)/2            & =    \\
		{\binom{e-1}{2}}-(e^2 +e(-1+  m(a-2)-a)+2m)/2       & \leq \\
		{\binom{e-1}{2}}-(e^2-3e+4)/2                       & =    \\
		(e^2-3e+2)/2-(e^2-3e+4)/2                           & <0   \\
	\end{align*}
	as $m\geq 2$ and $ a\geq 2$. This is impossible. Hence $m=1$.
	If $B_2$ is not primitive, then  \eqref{eq-d1} necessarily holds. Moreover,  we can apply the same argument with $B_2$ in place of $B_3$ and with $\mult_{p_{1,2}}(C)$ in place of $m=\mult_{p_{1,3}}(C)$. We thus get $\mult_{p_{1,2}}(C)=1=\mult_{p_{1,3}}(C)$, hence by \eqref{eq-d1}, we obtain $d_1=1+1=2$. This proves the first part of \eqref{3CSb}.

	For the last part of   \eqref{3CSb},  assume that $B_2$ is   primitive; then either  \eqref{eq-d1} or \eqref{eq-d2} holds.
	Hence, by the above argument,  for some $i\in \{1,2\}$ we have that $C$ has degree $d_i$ in the embedding given by $B_i$, and a point of multiplicity $d_i-1$ in some $p_{i,3}\in B_i\cap B_3$. The proof is complete.
\end{proof}

\begin{thm}
	\label{thm-nonprim}
	Let $B=B_1\cup B_2\cup B_3\subset S$ be a  3C-curve with  $B_i$   very ample  for $i=1,2,3$, and such that      $B_1$ is primitive and $B_2,B_3$ are not  primitive. Then
	$\cE(B)$ is finite, and    $(C\cdot B_1)\leq 2$ for every   $C\in \cE(B)$,
\end{thm}

\begin{proof}
	By Proposition~\ref{3CS} we have $\cE(B)=\Hyp(B,2)$.
	Let $C\in \Hyp_{\md}(B,2)$, with, as usual, $\md=(d_1,d_2,d_3)$ and $d_i=(C\cdot B_i)$.

	We view  $S$ in  projective space  embedded    by the linear system $|B_1|$.

	Suppose  $d_1=1$, then $C$ is a line in our embedding.
	If $C\cap B\subset N$ then $C$ is one of the finitely many lines  through two points of $N$.  More exactly, since $d_1=1$ we have only two cases, either $C\cap B=\{p_{1,2},p_{2,3}\}$ or  $C\cap B=\{p_{1,3},p_{2,3}\}$ (with the usual convention $p_{i,j}\in B_i\cap B_j$).
	Hence  we have at most $(B_1\cdot B_2)(B_2\cdot B_3)+(B_1\cdot B_3)(B_2\cdot B_3)$ such lines.

	If $C\cap B\not\subset N$ then   there are two indices, $i,j$ and $p_{i,j}\in B_i\cap B_j$ such that
	$$C\in  \Hyp_{(d_i,d_j)}(B_i\cup B_j, p_{i,j}).$$
	By  Proposition~\ref{S1} the set $\Hyp_{(d_i,d_j)}(B_i\cup B_j, p_{i,j})$ is not empty only if one between $B_i$ and $B_j$ is primitive,
	therefore we must have $i=1$. Assume, with no loss of generality,  $j=2$;  then     $C$ is hypertangent to $B_2$ at $p_{1,2}$ and hypertangent to $B_3$ in a point not lying on $N$; since  $B_2$ and $B_3$   are not primitive, they are not lines in our embedding.  Therefore such a $C$, if it exists,   is the unique line in $S$  tangent to $B_2$ at $p_{1,2}.$
	Since we have $(B_1\cdot B_2)$ choices for $p_{1,2}$, and we can argue in the same way if $j=3$, we get at most $(B_1\cdot B_2)+(B_1\cdot B_3)$ possibilities for such a $C$.

	Now let $d_1=2$.   By Proposition~\ref{3CS} the curve $C$ is a conic  hypertangent to $B_2$ at $p_{1,2}$ and to $B_3$ at  $p_{1,3}$.  Therefore (as $B_i$ and $C$ lie in the smooth surface $S$) the tangent lines of $C$ and $B_i$ at $p_{1,i}$ coincide for $i=2,3$, in symbols, $T_{p_{1,2}}C=T_{p_{1,2}}B_2$ and $T_{p_{1,3}}C=T_{p_{1,3}}B_3$. Let  $P\cong \PP^2$ be the plane containing $C$; as $P$ contains the  tangent lines to $C$, it contains the lines
	$T_{p_{1,2}}B_2$ and $T_{p_{1,3}}B_3$, which therefore have to be incident. But then   $P$ is uniquely determined by   $T_{p_{1,2}}B_2$ and $T_{p_{1,3}}B_3$. Since the intersection of $P$ with $S$ contains only finitely many conics, we have finitely many curves such that $d_1=2$.

	By Proposition~\ref{3CS} \eqref{3CSa} the set  $\Hyp_{\md}(B,2)$ is empty if $d_1\geq 3$, so we are done\end{proof}

\begin{remark}\label{rk-gt}	 The hypotheses in Theorem~\ref{thm-nonprim} imply that   $(S,B)$ is of log general type. This follows from
	Fujita's Conjecture for surfaces (a well-known consequence of results of Reider \cite{Reider}, see \cite[Section 10.4]{Lazarsfeld2}).
	Indeed, as at least one component of $B$ is not primitive,  we have $K_S+B=K_S+\sum_{i=1}^4H_i$ with $H_i$ very ample for all $i$. Now Fujita's Conjecture states that $K_S+4H_i$ is very ample, hence $4(K_S+B)$  is also very ample, hence $K_S+B$ is ample.

	The same holds for Theorems~\ref{th:main} and \ref{th:hyperbolic}, since  this argument works if we replace ``4" with    $n\geq 4$. 	 \end{remark}
Comparing with the results obtained in \cite{CT} for $S = \PP^2$ it is natural to consider the following questions for further investigation.

\begin{question}

	Let $B=B_1\cup B_2\cup B_3\subset S$ be a  3C-curve with $B_i$ very ample for all $i=1,2,3$, $B_1$ and $B_2$ primitive, and $B_3$ not primitive. Is $\cE(B)$ finite?
\end{question}

If $S = \PP^2$ the answer is yes, see \cite[Proposition 3.2.1 and Theorem 3.3.2]{CT}. Moreover, on a slightly different vein, \cite[Theorem 3.3.4]{CT} shows that the answer to the next question is affirmative for $\PP^2$.

\begin{question}

	Let $B=B_1\cup B_2\cup B_3\subset S$ be a  3C-curve with $B_i$ very ample for all $i=1,2,3$, $B_1$ primitive, and $B_2, B_3$ not primitive. If the curve $B$ is general, is $\cE(B)$ empty?
\end{question}

Here   general   means that the components $B_1,B_2,B_3$ vary in a (Zariski) open subset of  $|B_1|\times |B_2|\times |B_3|$, where $|B_i|$ is the linear system associated to $B_i$.

\bibliographystyle{alpha}
\bibliography{references.bib}

\end{document}